\documentclass{amsart}
\usepackage{amsfonts}
\usepackage{amsmath}
\usepackage{amssymb}

\newtheorem{theorem}{Theorem}

\newtheorem{corollary}[theorem]{Corollary}

\newtheorem{lemma}[theorem]{Lemma}

\newtheorem{proposition}[theorem]{Proposition}
\newtheorem{remark}[theorem]{Remark}

\begin{document}

\author{Maciej \textsc{Ciesielski} and Ryszard \textsc{P\l uciennik}}
\title[On some modifications of $n$-th von Neumann-Jordan
constant]{On some modifications of $n$-th von Neumann-Jordan
constant for Banach spaces}

\maketitle

\begin{abstract}
There are discussed some modifications of $n$-th von Neumann-Jordan constant
considered by M. Kato, Y. Takahashi and K. Hashimoto in \cite{KT3}. We
study, among others, upper, lower, upper modified and lower modified $n$-th
von Neumann-Jordan constant and relationships between them. There are
characterized uniformly non-$l_{n}^{1}$ Banach spaces in terms of the upper
modified $n$-th von Neumann-Jordan constant. Moreover, this constant is
calculated explicitly for Lebesgue spaces $L^{p}$ and $l^{p}$ $(1\leq p\leq
\infty ).$
\end{abstract}

\bigskip\ 

{\small \underline{2010 Mathematics Subjects Classification:}\hspace{1.5cm}\
\ \ \quad\ \quad 46E30, 46E40, 46B20. }\smallskip\ 

{\small \underline{Key Words and Phrases:}\hspace{0.15in}von Neumann-Jordan
constant, modified }${\small n}${\small -th von Neumann-Jordan constant,
uniformly non-}$l_{n}^{(1)}${\small -Banach space, }$B${\small -convex
Banach space. }

\bigskip\ \ 

\begin{center}
\textbf{1. Preliminaries}\ 
\end{center}

Let $S(X)$ (resp. $B(X))$ be the unit sphere (resp. the unit closed ball) of
a real Banach space $(X,\left\Vert \cdot \right\Vert _{X}).$ The letters $%
\mathbb{Z},$ $\mathbb{N}$ and $\mathbb{R}$ stand for the sets of integers,
positive integers and real numbers, respectively. For any subset $A\subset
X, $ denote $A^{n}=\underset{n\text{-times}}{\underbrace{A\times ...\times A}%
}.$

In 1937 Clakson \cite{C1}, on the basis of the famous paper \cite{JN} by
Jordan and von Neumann, introduced the constant $C_{NJ}(X)$ (called the 
\textit{von Neumann-Jordan constant or NJ-constant for short}) as the
smallest constant $C\geq 1$ such that 
\begin{equation*}
\frac{1}{C}\leq \frac{\left\Vert x+y\right\Vert _{X}^{2}+\left\Vert
x-y\right\Vert _{X}^{2}}{2\left( \left\Vert x\right\Vert _{X}^{2}+\left\Vert
y\right\Vert _{X}^{2}\right) }\leq C
\end{equation*}%
holds for any $x,y\in X$ with $\left\Vert x\right\Vert _{X}^{2}+\left\Vert
y\right\Vert _{X}^{2}>0.$ An equivalent and more convenient definition of
NJ-constant is given in \cite{KMT} by the formula 
\begin{equation*}
C_{NJ}(X)=\sup \left\{ \frac{\left\Vert x+y\right\Vert _{X}^{2}+\left\Vert
x-y\right\Vert _{X}^{2}}{2\left( \left\Vert x\right\Vert _{X}^{2}+\left\Vert
y\right\Vert _{X}^{2}\right) }:x\in S(X),\text{ }y\in B(X)\right\} .
\end{equation*}

The classical Jordan and von Neumann results \cite{JN} state that $1\leq
C_{NJ}(X)\leq 2$ for any Banach space $X$ and $C_{NJ}(X)=1$ if and only if $%
X $ is a Hilbert space. Clarkson \cite{C1} showed that if $1\leq p\leq
\infty $ and $\dim L^{p}(\mu )\geq 2,$ then $C_{NJ}(L^{p}(\mu ))=2^{2/\min
\{p,q\}-1}, $ where $1/p+1/q=1.$ Kato and Takahashi \cite{KT2}, observed
that $C_{NJ}(X)=C_{NJ}(X^{\ast }).$ Moreover they proved that if the Banach
space $X$ is uniformly convex, then $C_{NJ}(X)<2$ and if $C_{NJ}(X)<2,$ then 
$X$ admits an equivalent uniformly convex norm. This same authors \cite{KT2}
state that the Banach space $X$ is uniformly non-square if and only if $%
C_{NJ}(X)<2.$

A similar constant 
\begin{equation*}
C_{NJ}^{^{\prime }}(X)=\sup \left\{ \frac{\left\Vert x+y\right\Vert
_{X}^{2}+\left\Vert x-y\right\Vert _{X}^{2}}{4}:x,y\in S(X)\right\}
\end{equation*}%
was introduced in 2006 by Gao \cite{Gao} and called the modified von
Neumann-Jordan constant. It is clear that $C_{NJ}^{^{\prime }}(X)\leq
C_{NJ}(X).$ It has been shown that $C_{NJ}^{^{\prime }}(X)$ does not
necessarily coincide with $C_{NJ}(X)$ (see \cite{AMP, GS}).

To generalize the von Neumann-Jordan constant, denote 
\begin{equation*}
C^{(n)}(x_{1},x_{2},...,x_{n})=\frac{\sum_{\theta _{j}=\pm 1}\left\Vert
x_{1}+\sum_{j=2}^{n}\theta _{j}x_{j}\right\Vert _{X}^{2}}{%
2^{n-1}\sum_{j=1}^{n}\left\Vert x_{j}\right\Vert _{X}^{2}}
\end{equation*}%
for any $x_{1},x_{2},...,x_{n}\in X$ such that $\sum_{j=1}^{n}\left\Vert
x_{j}\right\Vert _{X}^{2}>0.$ \smallskip\ 

\textbf{Definition 1.} The smallest (the largest) constant $C>0$ such that 
\begin{equation}
C^{(n)}(x_{1},x_{2},...,x_{n})\leq C\text{ }\left( C\leq
C^{(n)}(x_{1},x_{2},...,x_{n})\right)  \label{Cn}
\end{equation}%
for all $x_{j}\in X,$ $(j=1,2,...,n$ and $n\geq 2)$ with $%
\sum_{j=1}^{n}\left\Vert x_{j}\right\Vert _{X}^{2}>0$ is called an \textit{%
upper} (\textit{lower}) $n$-\textit{th von Neumann-Jordan constant} and it
is denoted by $\overline{C}_{NJ}^{(n)}(X)$ $\left( \underline{C}%
_{NJ}^{(n)}(X)\right) ,$ respectively. If the infimum (supremum) of $C$
satisfying (\ref{Cn}) is taken over all $x_{j}\in S(X),$ $(j=1,2,...,n$ and $%
n\geq 2),$ then it is called \textit{upper} (\textit{lower}) \textit{modified%
} $n$-\textit{th von Neumann-Jordan constant} and it is denoted by $%
\overline{C}_{mNJ}^{(n)}(X)$ $\left( \underline{C}_{mNJ}^{(n)}(X)\right) ,$
respectively.\ \smallskip\ 

It is well known that $\overline{C}_{NJ}^{(2)}(X)=\left[ \underline{C}%
_{NJ}^{(2)}(X)\right] ^{-1}=C_{NJ}(X)$ (see \cite{TK}). As it is proved
below, the equality $\overline{C}_{NJ}^{(n)}(X)=\left[ \underline{C}%
_{NJ}^{(n)}(X)\right] ^{-1}$ is not true in general for $n>2.$ Moreover, $%
\overline{C}_{mNJ}^{(2)}(X)=C_{NJ}^{\prime }(X)$ and $\overline{C}%
_{mNJ}^{(n)}(X)$ need not be equal to $\left[ \underline{C}_{mNJ}^{(n)}(X)%
\right] ^{-1}$ even for $n=2$ (see \cite{Gao}). The $n$-th von
Neumann-Jordan constant introduced and investigated by Kato, Takahashi and
Hashimoto in \cite{KT3} is exactly the upper $n$-th von Neumann-Jordan
constant. \smallskip\ 

In 1964 James \cite{Ja} introduced the notion of uniformly non-$l_{n}^{1}$
Banach space. Namely, a Banach space $X$ is called \textit{uniformly non}-$%
l_{n}^{1}$ if there exists $\delta >0$ such that for each $n$ elements of
the unit ball $B(X)$ 
\begin{equation*}
\min_{\theta _{j}=\pm 1}\left\Vert x_{1}+\sum\nolimits_{j=2}^{n}\theta
_{j}x_{j}\right\Vert _{X}\leq n\left( 1-\delta \right)
\end{equation*}%
(see \cite{GJ}). The definition remains the same if we replace the unit ball 
$B(X)$ by the unit sphere $S(X)$. If $X$ is uniformly non-$l_{n}^{1}$ for $%
n=2,$ then it is called \textit{uniformly non-square. }In 1987 Kami\'{n}ska
and Turett \cite{KaT} proved that the uniform non-$l_{n}^{1}$ for Banach
spaces is equivalent to the fact that there exists $\delta >0$ such that for
all $x_{1},x_{2},...,x_{n}$ in $X$ 
\begin{equation*}
\min_{\theta _{j}=\pm 1}\left\Vert x_{1}+\sum\nolimits_{j=2}^{n}\theta
_{j}x_{j}\right\Vert _{X}\leq \left( 1-\frac{\delta n\min_{1\leq i\leq
n}\left\Vert x_{i}\right\Vert _{X}}{\sum_{j=1}^{n}\left\Vert
x_{j}\right\Vert _{X}}\right) \sum_{j=1}^{n}\left\Vert x_{j}\right\Vert _{X}.
\end{equation*}

Banach spaces that are uniformly non-$l_{n}^{1}$ for a certain $n\in \mathbb{%
N}$ have been studied by A. Beck \cite{B}. Such spaces are said to be $B$-%
\textit{convex}. Beck \cite{B} proved that a Banach space $X$ is $B$-convex
if and only if a certain strong law of large numbers is valid for random
variables with ranges in $X.$ Moreover, $B$-convexity is a very important
property in fixed point theory because every $B$-convex uniformly monotone K%
\"{o}the space has the fixed point property \cite{AK}. \bigskip\ 

\begin{center}
\textbf{2. Basic properties}\ 
\end{center}

\begin{proposition}
\label{propCn}Let $n\geq 2$ and $X$ be a Banach space. The lower, upper,
modified lower and modified upper $n$-th von Neumann-Jordan constants have
the following properties:\ \newline
\hspace*{0.3in}$(a)$ $1\leq \overline{C}_{mNJ}^{(n)}(X)\leq \overline{C}%
_{NJ}^{(n)}(X)\leq n$ and $1/n\leq \underline{C}_{NJ}^{(n)}(X)\leq 
\underline{C}_{mNJ}^{(n)}(X)\leq 1;$ \newline
\hspace*{0.3in}$(b)$ $\overline{C}_{NJ}^{(n)}(X)\leq \overline{C}%
_{NJ}^{(n+1)}(X)$ and $\underline{C}_{NJ}^{(n+1)}(X)\leq \underline{C}%
_{NJ}^{(n)}(X);$ \newline
\hspace*{0.3in}$(c)$ $\overline{C}_{mNJ}^{(n)}(X)\leq \frac{n+1}{n}\overline{%
C}_{mNJ}^{(n+1)}(X)$ and $\underline{C}_{mNJ}^{(n+1)}(X)\leq \frac{n+1}{n}%
\underline{C}_{mNJ}^{(n)}(X)+\frac{1}{n+1}.$ \ 
\end{proposition}

\begin{proof}
Let $(X,\left\Vert \cdot \right\Vert _{X})$ be a Banach
space and $n\geq 2.$

$a)$ The estimation $\overline{C}_{NJ}^{(n)}(X)\leq n$ is proved in \cite%
{KT3}. $\overline{C}_{mNJ}^{(n)}(X)\leq \overline{C}_{NJ}^{(n)}(X)$ by the
definition. Putting $x_{1}\in S(X)$ and $x_{i}=x_{1}$ for $i=2,3,...,n,$ we
have 
\begin{eqnarray*}
C^{(n)}(x_{1},x_{1},...,x_{1}) &=&\frac{1}{n2^{n-1}}\sum_{\theta _{j}=\pm
1}\left\Vert x_{1}+\sum\nolimits_{j=2}^{n}\theta _{j}x_{1}\right\Vert
_{X}^{2} \\
&=&\frac{1}{n2^{n-1}}\sum\nolimits_{j=0}^{n-1}\binom{n-1}{j}\left\Vert
(n-2j)x_{1}\right\Vert _{X}^{2} \\
&=&\frac{1}{n2^{n-1}}\sum\nolimits_{j=0}^{n-1}\binom{n-1}{j}(n-2j)^{2}=1.
\end{eqnarray*}%
Hence 
\begin{equation*}
\overline{C}_{mNJ}^{(n)}(X)=\sup \left\{
C^{(n)}(x_{1},x_{2},...,x_{n}):x_{i}\in S(X),\text{ }i=1,2,...,n\right\}
\geq 1
\end{equation*}%
and 
\begin{equation*}
\underline{C}_{mNJ}^{(n)}(X)=\inf \left\{
C^{(n)}(x_{1},x_{2},...,x_{n}):x_{i}\in S(X),\text{ }i=1,2,...,n\right\}
\leq 1.
\end{equation*}%
Obviously, $\underline{C}_{NJ}^{(n)}(X)\leq \underline{C}_{mNJ}^{(n)}(X)$ by
the definition. To prove that $\frac{1}{n}\leq \underline{C}_{NJ}^{(n)}(X),$
we use the mathematical induction principle. For $n=2$ we have 
\begin{equation*}
\underline{C}_{NJ}^{(2)}(X)=\frac{1}{\overline{C}_{NJ}^{(2)}(X)}\geq \frac{1%
}{2}.
\end{equation*}%
Suppose that \ $\underline{C}_{NJ}^{(n-1)}(X)\geq \frac{1}{n-1}.$ Notice
that 
\begin{equation}
\left\Vert x+y\right\Vert _{X}^{2}+\left\Vert x-y\right\Vert _{X}^{2}\geq
2\left( \max \left\{ \left\Vert x\right\Vert _{X},\left\Vert y\right\Vert
_{X}\right\} \right) ^{2}\geq \left\Vert x\right\Vert _{X}^{2}+\left\Vert
y\right\Vert _{X}^{2}  \label{Inorm}
\end{equation}%
for any $x,y\in X.$ Really, since 
$$\left\Vert x+y\right\Vert _{X}+\left\Vert
x-y\right\Vert _{X}\geq 2\max \left\{ \left\Vert x\right\Vert
_{X},\left\Vert y\right\Vert _{X}\right\} =2m_{x,y}\geq\left\Vert x+y\right\Vert _{X}-\left\Vert x-y\right\Vert _{X}
$$ 
we have 
\begin{eqnarray*}
\left\Vert x+y\right\Vert _{X}^{2}+\left\Vert x-y\right\Vert _{X}^{2} &\geq
&\left\Vert x+y\right\Vert _{X}^{2}+\left( 2m_{x,y}-\left\Vert
x+y\right\Vert _{X}\right) ^{2} \\
&=&2\left\Vert x+y\right\Vert _{X}^{2}-4m_{x,y}\left\Vert x+y\right\Vert
_{X}+4\left( m_{x,y}\right) ^{2} \\
&=&2\left[ \left( \left\Vert x+y\right\Vert _{X}-m_{x,y}\right) ^{2}+\left(
m_{x,y}\right) ^{2}\right] \\
&\geq &2\left( \max \left\{ \left\Vert x\right\Vert _{X},\left\Vert
y\right\Vert _{X}\right\} \right) ^{2}\geq \left\Vert x\right\Vert
_{X}^{2}+\left\Vert y\right\Vert _{X}^{2}.
\end{eqnarray*}%
Hence 
\begin{eqnarray*}
C^{(n)}(x_{1},x_{2},...,x_{n}) &=&\frac{\sum_{\theta _{j}=\pm 1}\left\Vert
x_{1}+\sum_{j=2}^{n}\theta _{j}x_{j}\right\Vert _{X}^{2}}{%
2^{n-1}\sum_{j=1}^{n}\left\Vert x_{j}\right\Vert _{X}^{2}} \\
&=&\frac{\sum_{\theta _{j}=\pm 1}\left\Vert \sum_{j=1}^{n}\theta
_{j}x_{j}\right\Vert _{X}^{2}}{2^{n}\sum_{j=1}^{n}\left\Vert
x_{j}\right\Vert _{X}^{2}} \\
&=&\frac{\sum_{\theta _{j}=\pm 1}\left( \left\Vert \left( \sum_{j\neq
i}\theta _{j}x_{j}\right) +x_{i}\right\Vert _{X}^{2}+\left\Vert \left(
\sum_{j\neq i}\theta _{j}x_{j}\right) -x_{i}\right\Vert _{X}^{2}\right) }{%
2^{n}\sum_{j=1}^{n}\left\Vert x_{j}\right\Vert _{X}^{2}} \\
&\geq &\frac{\sum_{\theta _{j}=\pm 1}\left\Vert \sum_{j\neq i}\theta
_{j}x_{j}\right\Vert _{X}^{2}+2^{n-1}\left\Vert x_{i}\right\Vert _{X}^{2}}{%
2^{n}\sum_{j=1}^{n}\left\Vert x_{j}\right\Vert _{X}^{2}} \\
&\geq &\frac{\frac{2^{n-1}}{n-1}\sum_{j\neq i}\left\Vert x_{j}\right\Vert
_{X}^{2}+2^{n-1}\left\Vert x_{i}\right\Vert _{X}^{2}}{2^{n}\sum_{j=1}^{n}%
\left\Vert x_{j}\right\Vert _{X}^{2}} \\
&=&\frac{\frac{1}{n-1}\sum_{j\neq i}\left\Vert x_{j}\right\Vert
_{X}^{2}+\left\Vert x_{i}\right\Vert _{X}^{2}}{2\sum_{j=1}^{n}\left\Vert
x_{j}\right\Vert _{X}^{2}} \\
&=&\frac{\frac{1}{n-1}\sum_{j=1}^{n}\left\Vert x_{j}\right\Vert _{X}^{2}+%
\frac{n-2}{n-1}\left\Vert x_{i}\right\Vert _{X}^{2}}{2\sum_{j=1}^{n}\left%
\Vert x_{j}\right\Vert _{X}^{2}} \\
&=&\frac{1}{2\left( n-1\right) }+\frac{(n-2)\left\Vert x_{i}\right\Vert
_{X}^{2}}{2\left( n-1\right) \sum_{j=1}^{n}\left\Vert x_{j}\right\Vert
_{X}^{2}}
\end{eqnarray*}%
for any $i=1,2,...,n.$ It follows that 
\begin{eqnarray*}
C^{(n)}(x_{1},x_{2},...,x_{n}) &\geq &\frac{1}{n}\sum\limits_{i=1}^{n}\left( 
\frac{1}{2\left( n-1\right) }+\frac{(n-2)\left\Vert x_{i}\right\Vert _{X}^{2}%
}{2\left( n-1\right) \sum_{j=1}^{n}\left\Vert x_{j}\right\Vert _{X}^{2}}%
\right) \\
&=&\frac{1}{2\left( n-1\right) }+\frac{n-2}{2n\left( n-1\right) }=\frac{1}{n}
\end{eqnarray*}%
and consequently $\underline{C}_{NJ}^{(n)}(X)\geq \frac{1}{n},$ which
finishes the proof of $(a).$ \medskip\ 

$(b)$ The inequality $\overline{C}_{NJ}^{(n)}(X)\leq \overline{C}%
_{NJ}^{(n+1)}(X)$ is proved in \cite{KT3}. To prove the second inequality it
is enough to notice that 
\begin{equation*}
C^{(n)}(x_{1},x_{2},...,x_{n})=C^{(n+1)}(x_{1},x_{2},...,x_{n},0)
\end{equation*}%
for any elements $x_{1},x_{2},...,x_{n}\in X.$ Hence 
\begin{eqnarray*}
\underline{C}_{NJ}^{(n+1)}(X) &=&\inf \left\{
C^{(n+1)}(x_{1},x_{2},...,x_{n},x_{n+1}):x_{1},x_{2},...,x_{n+1}\in X\right\}
\\
&\leq &\inf \left\{ C^{(n)}(x_{1},x_{2},...,x_{n}):x_{1},x_{2},...,x_{n}\in
X\right\} =\underline{C}_{NJ}^{(n)}(X).
\end{eqnarray*}

$(c)$ For any $x_{1},x_{2},...,x_{n+1}\in S(X),$ by the inequality (\ref%
{Inorm}), we have 
\begin{eqnarray*}
C^{(n+1)}(x_{1},x_{2},...,x_{n+1}) &=&\frac{\sum\limits_{\theta _{j}=\pm
1}\left( \left\Vert \left( \sum\limits_{j=1}^{n}\theta _{j}x_{j}\right)
+x_{n+1}\right\Vert _{X}^{2}+\left\Vert \left( \sum\limits_{j=1}^{n}\theta
_{j}x_{j}\right) -x_{n+1}\right\Vert _{X}^{2}\right) }{(n+1)2^{n+1}} \\
&\geq &\frac{1}{(n+1)2^{n}}\sum_{\theta _{j}=\pm 1}\max \left\{ \left\Vert
\sum\nolimits_{j=1}^{n}\theta _{j}x_{j}\right\Vert _{X}^{2},\left\Vert
x_{n+1}\right\Vert _{X}^{2}\right\} \\
&\geq &\frac{1}{(n+1)2^{n}}\sum_{\theta _{j}=\pm 1}\left\Vert
\sum\nolimits_{j=1}^{n}\theta _{j}x_{j}\right\Vert _{X}^{2}=\frac{n}{n+1}%
C^{(n)}(x_{1},x_{2},...,x_{n}),
\end{eqnarray*}%
whence $\overline{C}_{mNJ}^{(n)}(X)\leq \frac{n+1}{n}\overline{C}%
_{mNJ}^{(n+1)}(X).$

By the Clarkson inequality for $p=2,$ we have 
\begin{eqnarray*}
C^{(n+1)}(x_{1},x_{2},...,x_{n+1}) &=&\frac{\sum\limits_{\theta _{j}=\pm
1}\left( \left\Vert \left( \sum\limits_{j=1}^{n}\theta _{j}x_{j}\right)
+x_{n+1}\right\Vert _{X}^{2}+\left\Vert \left( \sum\limits_{j=1}^{n}\theta
_{j}x_{j}\right) -x_{n+1}\right\Vert _{X}^{2}\right) }{(n+1)2^{n+1}} \\
&\leq &\frac{2\sum_{\theta _{j}=\pm 1}\left\Vert
\sum\nolimits_{j=1}^{n}\theta _{j}x_{j}\right\Vert _{X}^{2}}{(n+1)2^{n+1}}+%
\frac{2^{n+1}\left\Vert x_{n+1}\right\Vert _{X}^{2}}{(n+1)2^{n+1}} \\
&=&\frac{n+1}{n}C^{(n)}(x_{1},x_{2},...,x_{n})+\frac{1}{n+1}
\end{eqnarray*}%
for any $x_{1},x_{2},...,x_{n+1}\in S(X)$ and consequently $\underline{C}%
_{mNJ}^{(n+1)}(X)\leq \frac{n+1}{n}\underline{C}_{mNJ}^{(n)}(X)+\frac{1}{n+1}%
.$
\end{proof}

\begin{proposition}
\label{Hilbert}Let $n\geq 2$ and $X$ be a Banach space. \newline
\hspace*{0.3in}$(a)$ The following conditions are equivalent: \newline
\hspace*{0.3in}\hspace*{0.3in}$(i)$ $X$ is a Hilbert space. \newline
\hspace*{0.3in}\hspace*{0.3in}$(ii)$ $\overline{C}_{NJ}^{(n)}(X)=1.$ \newline
\hspace*{0.3in}\hspace*{0.3in}$(iii)$ $\underline{C}_{NJ}^{(n)}(X)=1.$ 
\newline
\hspace*{0.3in}$(b)$ If $X$ is a Hilbert space, then $\overline{C}%
_{mNJ}^{(n)}(X)=\underline{C}_{mNJ}^{(n)}(X)=1.$ \medskip\ 
\end{proposition}

\begin{proof}

$(a)$ By Theorem $5$ $(iii)$ in \cite{KT3}, conditions $(i)$
and $(ii)$ are equivalent. To prove the implication $(i)\Rightarrow (iii)$
suppose that $X$ is a Hilbert space. By elementary calculations, we get that 
\begin{equation}
C^{(n)}(x_{1},x_{2},...,x_{n})=1  \label{CnH}
\end{equation}%
for any elements $x_{1},x_{2},...,x_{n}\in X,$ whence $\underline{C}%
_{NJ}^{(n)}(X)=1.$ Conversely, $\underline{C}_{NJ}^{(n)}(X)=1,$ then, by
Proposition \ref{propCn} $(a)$ and $(b),$ we have 
\begin{equation*}
1\geq \frac{1}{C_{NJ}(X)}=\underline{C}_{NJ}^{(2)}(X)\geq \underline{C}%
_{NJ}^{(n)}(X)=1.
\end{equation*}%
Hence $C_{NJ}(X)=1.$ Consequently, $X$ is a Hilbert space (see \cite{JN}).

$(b)$ follows immediately from (\ref{CnH}). 
\end{proof}

In general, the explicit calculation of various types of $n$-th von
Neumann-Jordan constant is rather hard problem. Anyway, the next proposition
can be helpful to do this.

\begin{proposition}
\label{D}\textbf{\ }Let $(X,\left\Vert \cdot \right\Vert _{X})$ be a Banach
space and $n\geq 2.$ Denote $D_{1}=\left[ B(X)\right] ^{n}\setminus \{%
\mathbf{0}\},$ $D_{2}=B(l_{n}^{2}(X))\setminus \{\mathbf{0}\},$ $%
D_{3}=S(l_{n}^{2}(X)),$ where $\mathbf{0}=(0,0,...,0).$ Then 
\begin{equation}
\overline{C}_{NJ}^{(n)}(X)=\sup \left\{
C^{(n)}(x_{1},x_{2},...,x_{n}):\left( x_{1},x_{2},...,x_{n}\right) \in
D_{j}\right\}  \label{sp}
\end{equation}%
and 
\begin{equation}
\underline{C}_{NJ}^{(n)}(X)=\inf \left\{
C^{(n)}(x_{1},x_{2},...,x_{n}):\left( x_{1},x_{2},...,x_{n}\right) \in
D_{j}\right\}  \label{if}
\end{equation}%
for any $j=1,2,3.$\ 
\end{proposition}

\begin{proof}
Since 
\begin{equation*}
S(l_{n}^{2}(X))\subset B(l_{n}^{2}(X))\setminus \{\mathbf{0}\}\subset \left[
B(X)\right] ^{n}\setminus \{\mathbf{0}\}\subset X^{n}\setminus \{\mathbf{0}%
\},
\end{equation*}%
it follows that 
\begin{equation*}
\sup_{\mathbf{x}\in D_{3}}C^{(n)}(\mathbf{x})\leq \sup_{\mathbf{x}\in
D_{2}}C^{(n)}(\mathbf{x})\leq \sup_{\mathbf{x}\in D_{1}}C^{(n)}(\mathbf{x}%
)\leq \overline{C}_{NJ}^{(n)}(X),
\end{equation*}%
where $\mathbf{x}=(x_{1},x_{2},...,x_{n}).$ To show (\ref{sp}), it remains
to prove that $\sup_{\mathbf{x}\in D_{3}}C^{(n)}(\mathbf{x})\geq \overline{C}%
_{NJ}^{(n)}(X).$ Let $\mathbf{x}=\left( x_{1},x_{2},...,x_{n}\right) \in
X^{n}\setminus \{\mathbf{0}\}.$ Define the sequence $\mathbf{y}%
=(y_{k})_{k=1}^{n}$ by 
\begin{equation*}
y_{k}=\frac{x_{k}}{\left( \sum_{j=1}^{n}\left\Vert x_{j}\right\Vert
_{X}^{2}\right) ^{\frac{1}{2}}}
\end{equation*}%
for $k=1,2,...,n$ and $n\geq 2.$ Obviously, $\mathbf{y}\in S(l_{n}^{2}(X)).$
Hence 
\begin{eqnarray*}
\sup_{\mathbf{x}\in D_{3}}C^{(n)}(\mathbf{x}) &\geq &C^{(n)}(\mathbf{y})=%
\frac{\sum_{\theta _{j}=\pm 1}\left\Vert y_{1}+\sum_{j=2}^{n}\theta
_{j}y_{j}\right\Vert _{X}^{2}}{2^{n-1}} \\
&=&\frac{\sum_{\theta _{j}=\pm 1}\left\Vert x_{1}+\sum_{j=2}^{n}\theta
_{j}x_{j}\right\Vert _{X}^{2}}{2^{n-1}\sum_{j=1}^{n}\left\Vert
x_{j}\right\Vert _{X}^{2}}=C^{(n)}(\mathbf{x})
\end{eqnarray*}%
for any elements $\mathbf{x}\in X^{n}\setminus \{\mathbf{0}\}.$ Therefore 
\begin{equation*}
\sup_{\mathbf{x}\in D_{3}}C^{(n)}(\mathbf{x})\geq \sup \left\{ C^{(n)}(%
\mathbf{x}):\mathbf{x}\in X^{n}\setminus \{\mathbf{0}\}\right\} =\overline{C}%
_{NJ}^{(n)}(X)
\end{equation*}%
which finishes the proof of (\ref{sp}). The equality (\ref{if}) can be
proved similarly. 
\end{proof}

Let $n\geq 2.$ Define 
\begin{equation*}
A_{2}=\left[ 
\begin{array}{cc}
1 & 1 \\ 
1 & -1%
\end{array}%
\right] _{2\times 2}
\end{equation*}%
and for each integers $n>2$%
\begin{equation*}
A_{n}=\left[ 
\begin{array}{cc}
A_{n-1} & \mathbf{1} \\ 
A_{n-1} & \mathbf{-1}%
\end{array}%
\right] _{2^{n-1}\times n},
\end{equation*}%
where $\mathbf{1}$ denotes the $2^{n-2}$-by-$1$ column vector in which all
the elements are equal to $1.$ The matrix $A_{n}$ generates a linear
operator 
\begin{equation*}
T_{n}:l_{n}^{2}(X)\rightarrow l_{2^{n-1}}^{2}(X)
\end{equation*}%
defined for any $x\in l_{n}^{2}(X)$ by the formula 
\begin{equation*}
T_{n}(x)=A_{n}x.
\end{equation*}

A one-to-one correspondence between $n$-th von Neumann-Jordan constant $%
\overline{C}_{NJ}^{(n)}(X)$ and the norm of the operator $T_{n}$ is given by
the following result.

\begin{corollary}
\label{T} Let $(X,\left\Vert \cdot \right\Vert _{X})$ be a Banach space and $%
T_{n}:l_{n}^{2}(X)\rightarrow l_{2^{n-1}}^{2}(X)$ be the linear operator
generated by the matrix $A_{n}.$\ Then 
\begin{equation*}
\overline{C}_{NJ}^{(n)}(X)=\frac{||T_{n}||^{2}}{2^{n-1}}
\end{equation*}%
for any integer $n\geq 2.$ \ 
\end{corollary}

\begin{proof} 
Fix $n\geq 2.$ Let $x_{1},x_{2},...,x_{n}\in X$ and $%
\sum_{j=1}^{n}\left\Vert x_{j}\right\Vert _{X}^{2}>0.$ Denote $%
x=(x_{1},x_{2},...,x_{n}).$ By Proposition \ref{D} $\eqref{sp}$ we have 
\begin{eqnarray*}
\overline{C}_{NJ}^{(n)}(X) &=&\overline{F}^{(n)}=\sup \left\{ \frac{%
\sum_{\theta _{j}=\pm 1}\left\Vert x_{1}+\sum_{j=2}^{n}\theta
_{j}x_{j}\right\Vert _{X}^{2}}{2^{n-1}\sum_{j=1}^{n}\left\Vert
x_{j}\right\Vert _{X}^{2}}:x\in S(l_{n}^{2}(X))\right\} \\
&=&\frac{1}{2^{n-1}}\sup \left\{ ||T_{n}x||_{l_{2^{n-1}}^{2}(X)}^{2}:x\in
S(l_{n}^{2}(X))\right\} =\frac{||T_{n}||^{2}}{2^{n-1}}.
\end{eqnarray*}
\end{proof} 

\begin{corollary}
\label{Dual}Let $\left( X^{\ast },\left\Vert \cdot \right\Vert _{X^{\ast
}}\right) $ be the dual space of the Banach space $(X,\left\Vert \cdot
\right\Vert _{X}).$ Then \newline
\hspace*{0.3in}$(a)$ $\underline{C}_{NJ}^{(n)}(X^{\ast })\geq \frac{1}{%
\overline{C}_{NJ}^{(n)}(X)}.$ \ \newline
\hspace*{0.3in}$(b)$ $\underline{C}_{NJ}^{(n)}(X)\geq \frac{1}{\overline{C}%
_{NJ}^{(n)}(X^{\ast })}.$\ 
\end{corollary}

\begin{proof} 
$(a)$ Define an operator $T_{n}:l_{n}^{2}(X)\rightarrow
l_{2^{n-1}}^{2}(X)$ as above. Let $T_{n}^{\ast }$ be the adjoint of operator 
$T_{n}.$ Obviously, $T_{n}^{\ast }:l_{2^{n-1}}^{2}(X^{\ast })\rightarrow
l_{n}^{2}(X^{\ast })$ is generated by the matrix $A_{n}^{\ast }=A_{n}^{T}.$
Then, by Corollary \ref{T}, we get 
\begin{equation}
\frac{1}{\overline{C}_{NJ}^{(n)}(X)}=\frac{2^{n-1}}{||T_{n}||^{2}}=\frac{%
2^{n-1}}{||T_{n}^{\ast }||^{2}}\leq \frac{2^{n-1}\left\Vert y^{\ast
}\right\Vert _{l_{2^{n-1}}^{2}(X^{\ast })}^{2}}{\left\Vert T_{n}^{\ast
}y^{\ast }\right\Vert _{l_{n}^{2}(X^{\ast })}^{2}}  \label{spr}
\end{equation}%
for any $y^{\ast }=(y_{1}^{\ast },y_{2}^{\ast },...,y_{2^{n-1}}^{\ast })\in
l_{2^{n-1}}^{2}(X^{\ast })\setminus \{\mathbf{0}\}.$ Let $S_{n}^{\ast }=%
\frac{1}{2^{n-1}}A_{n}.$ Then, $S_{n}:l_{n}^{2}(X^{\ast })\rightarrow
l_{2^{n-1}}^{2}(X^{\ast }).$ Since 
\begin{equation*}
A_{n}^{\ast }\cdot \frac{1}{2^{n-1}}A_{n}=I_{n},
\end{equation*}%
where $I_{n}$ is the identity matrix of size $n\times{n},$ it follows that 
\begin{equation*}
T_{n}^{\ast }\left( S_{n}^{\ast }x^{\ast }\right) =x^{\ast }
\end{equation*}%
for any $x^{\ast }=(x_{1}^{\ast },x_{2}^{\ast },...,x_{n}^{\ast })\in
l_{n}^{2}\left( X^{\ast }\right) .$ Hence$,$ by (\ref{spr}), we have 
\begin{eqnarray*}
\frac{1}{\overline{C}_{NJ}^{(n)}(X)} &\leq &\frac{2^{n-1}\left\Vert
S_{n}^{\ast }x^{\ast }\right\Vert _{l_{2^{n-1}}^{2}(X^{\ast })}^{2}}{%
\left\Vert T_{n}^{\ast }\left( S_{n}^{\ast }x^{\ast }\right) \right\Vert
_{l_{n}^{2}(X^{\ast })}^{2}}=\frac{\left\Vert A_{n}x^{\ast }\right\Vert
_{l_{2^{n-1}}^{2}(X^{\ast })}^{2}}{2^{n-1}\left\Vert x^{\ast }\right\Vert
_{l_{n}^{2}(X^{\ast })}^{2}} \\
&=&\frac{\sum_{\theta _{j}=\pm 1}\left\Vert x^{*}_{1}+\sum_{j=2}^{n}\theta
_{j}x^{*}_{j}\right\Vert _{X^{\ast }}^{2}}{2^{n-1}\sum_{j=1}^{n}\left\Vert
x^{*}_{j}\right\Vert _{X^{\ast }}^{2}}=C^{(n)}\left( x_{1}^{\ast },x_{2}^{\ast
},...,x_{n}^{\ast }\right)
\end{eqnarray*}%
for any $(x_{1}^{\ast },x_{2}^{\ast },...,x_{n}^{\ast })\in \left( X^{\ast
}\right) ^{n}\setminus \{\mathbf{0}\}.$ By the definition of the lower $n$%
-th von Neumann-Jordan constant, we get the thesis of $(a).$ \medskip\ 

$(b)$ Since $X$ can be isometrically embedded into $X^{\ast \ast },$ it
follows that $\underline{C}_{NJ}^{(n)}(X)\geq \underline{C}%
_{NJ}^{(n)}(X^{\ast \ast }).$ Hence, by $(a)$, we have 
\begin{equation*}
\underline{C}_{NJ}^{(n)}(X)\geq \underline{C}_{NJ}^{(n)}(X^{\ast \ast })\geq 
\frac{1}{\overline{C}_{NJ}^{(n)}(X^{\ast })}.
\end{equation*}
\end{proof}

Values of $n$-th von Neumann-Jordan constant some classical Banach spaces
gives the following.

\begin{proposition}
\label{CninL}Let $n\geq 2.$ \newline
\hspace*{0.3in}$(a)$ If $1\leq p\leq 2$ and $k=\dim l^{p}\geq n,$ then $%
\overline{C}_{NJ}^{(n)}(l^{p})=\overline{C}_{mNJ}^{(n)}(l^{p})=n^{\frac{2}{p}%
-1}.$ \newline
\hspace*{0.3in}$(b)$ If $\dim l^{\infty }\geq 2^{n-1},$ then $\overline{C}%
_{NJ}^{(n)}(l^{\infty })=\overline{C}_{mNJ}^{(n)}(l^{\infty })=n.$ \newline
\hspace*{0.3in}$(c)$ If $\dim l^{\infty }\geq n,$ then $\underline{C}%
_{NJ}^{(n)}(l^{\infty })=\underline{C}_{mNJ}^{(n)}(l^{\infty })=\frac{1}{n}.$
\newline
\hspace*{0.3in}$(d)$ Let $(\Omega ,\Sigma ,\mu )$ be a measure space with
non-atomic $\sigma $-finite measure $\mu .$ Then 
\begin{equation*}
\overline{C}_{mNJ}^{(n)}(L^{1}(\mu ))=\overline{C}_{NJ}^{(n)}(L^{1}(\mu ))=%
\overline{C}_{NJ}^{(n)}(L^{\infty }(\mu ))=\overline{C}_{NJ}^{(n)}(L^{\infty
}(\mu ))=n
\end{equation*}%
and 
\begin{equation*}
\underline{C}_{mNJ}^{(n)}(L^{\infty }(\mu ))=\underline{C}%
_{NJ}^{(n)}(L^{\infty }(\mu ))=\frac{1}{n}.
\end{equation*}
\end{proposition}

\begin{proof} 
$(a)$ Let $k=\dim l^{p}\geq n.$ By Theorem 3 $(ii)$ in \cite%
{KT3}, $\overline{C}_{NJ}^{(n)}(l^{p})=n^{\frac{2}{p}-1}.$ By Proposition %
\ref{propCn}\textbf{~}$(b),$ $\overline{C}_{mNJ}^{(n)}(l^{p})\leq n^{\frac{2%
}{p}-1}.$ Taking the canonical basis $\left( e_{i}\right) _{i=1}^{k}$ $%
(n\leq k\leq \infty )$ in $l^{p},$ we get 
\begin{eqnarray*}
C^{(n)}\left( e_{1},e_{2},...,e_{n}\right) &=&\frac{\sum_{\theta _{j}=\pm
1}\left\Vert e_{1}+\sum_{j=2}^{n}\theta _{j}e_{j}\right\Vert _{l^{p}}^{2}}{%
2^{n-1}\sum_{j=1}^{n}\left\Vert e_{j}\right\Vert _{l^{p}}^{2}}= \\
&=&\frac{n^{\frac{2}{p}}2^{n-1}}{n2^{n-1}}=n^{\frac{2}{p}-1}.
\end{eqnarray*}%
Since $e_{i}\in S(l^{p})$ for $i\in \mathbb{N}\cap \lbrack 1,k],$ it follows
that 
\begin{equation*}
\overline{C}_{mNJ}^{(n)}(l^{p})\geq C^{(n)}\left(
e_{1},e_{2},...,e_{n}\right) =n^{\frac{2}{p}-1}.
\end{equation*}%
Hence $\overline{C}_{mNJ}^{(n)}(l^{p})=n^{\frac{2}{p}-1}.$

$(b)$ Let $k=\dim l^{\infty }\geq 2^{n-1}.$ Then $l^{\infty }$ is not
uniformly non-$l_{n}^{1}.$ By Theorem $5$ $(iv)$ in \cite{KT3} and
Proposition \ref{propCn}\textbf{~}$(a),$ we conclude that $\overline{C}%
_{NJ}^{(n)}(l^{\infty })=n.$ To prove that $\overline{C}_{mNJ}^{(n)}(l^{%
\infty })=n$ take the matrix $A_{n}$ defined as above. The column $j$ of $%
A_{n}$ denote by $\mathbf{a}_{j}^{(n)}.$ Then $\mathbf{a}%
_{j}^{(n)}=[1,a_{2j}^{(n)},...,a_{2^{n-1}j}^{(n)}],$ where $a_{ij}^{(n)}=\pm
1$ for any $1\leq j\leq n$ and $2\leq i\leq 2^{n-1}.$ Define for $j=1,2,\dots,n$,
\begin{equation*}
z_{j}=\sum\limits_{i=1}^{2^{n-1}}a_{ij}^{(n)}e_{i},
\end{equation*}%
where $\left( e_{i}\right) _{i=1}^{k}$ $(n\leq k\leq \infty )$ is the
canonical basis in $l^{\infty }.$ Obviously, $\left\Vert z_{j}\right\Vert
_{l^{\infty }}=1.$ Let $(1,\theta _{2},...,\theta _{n})$ be an arbitrary
sequence such that $\theta _{j}=\pm 1$ for any $2\leq j\leq n.$ Then there
is exactly one row $i_{0}$ of $A_{n}$ such that 
\begin{equation*}
\left[ 1,\theta _{2},...,\theta _{n}\right] =\left[
a_{i_{0}1}^{(n)},a_{i_{0}2}^{(n)},...,a_{i_{0}n}^{(n)}\right] .
\end{equation*}%
Hence 
\begin{equation*}
1+\sum_{j=2}^{n}\theta _{j}a_{i_{0}j}^{(n)}=n.
\end{equation*}%
Moreover 
\begin{equation*}
\left\vert 1+\sum_{j=2}^{n}\theta _{j}a_{ij}^{(n)}\right\vert <n
\end{equation*}%
for any $i\not=i_{0}.$ Consequently, 
\begin{equation*}
\left\Vert z_{1}+\sum_{j=2}^{n}\theta _{j}z_{j}\right\Vert _{l^{\infty
}}=\max_{1\leq i\leq 2^{n-1}}\left\vert 1+\sum_{j=2}^{n}\theta
_{j}a_{ij}^{(n)}\right\vert =n,
\end{equation*}%
whence 
\begin{equation*}
C^{(n)}\left( z_{1},z_{2},...,z_{n}\right) =\frac{\sum_{\theta _{j}=\pm
1}\left\Vert z_{1}+\sum_{j=2}^{n}\theta _{j}z_{j}\right\Vert _{l^{\infty
}}^{2}}{2^{n-1}\sum_{j=1}^{n}\left\Vert z_{j}\right\Vert _{l^{\infty }}^{2}}=%
\frac{n^{2}2^{n-1}}{n2^{n-1}}=n.
\end{equation*}%
Therefore 
\begin{equation*}
n\leq \overline{C}_{mNJ}^{(n)}(l^{\infty })\leq \overline{C}%
_{NJ}^{(n)}(l^{\infty })=n,
\end{equation*}%
which completes the proof of $(b).$

$(c)$ Let $k=\dim l^{\infty }\geq n$ and $(e_{i})$ be the canonical basis
in $l^{\infty }.$ By Proposition \ref{propCn} $(a),$ we have 
\begin{eqnarray*}
\frac{1}{n} &\leq &\underline{C}_{NJ}^{(n)}(l^{\infty })\leq \underline{C}%
_{mNJ}^{(n)}(l^{\infty })\leq C^{(n)}\left( e_{1},e_{2},...,e_{n}\right) \\
&=&\frac{\sum_{\theta _{j}=\pm 1}\left\Vert e_{1}+\sum_{j=2}^{n}\theta
_{j}e_{j}\right\Vert _{l^{\infty }}^{2}}{2^{n-1}\sum_{j=1}^{n}\left\Vert
e_{j}\right\Vert _{l^{\infty }}^{2}}=\frac{1}{n},
\end{eqnarray*}%
whenever $\dim l^{\infty }\geq n.$

$(d)$ Since $L^{1}(\mu )$ contains an isometric copy of $l^{1},$ applying
Proposition \ref{CninL} $(a)$ for $p=1$ we get 
\begin{equation*}
n=\overline{C}_{mNJ}^{(n)}(l^{1})\leq \overline{C}_{mNJ}^{(n)}(L^{1}(\mu
))\leq \overline{C}_{NJ}^{(n)}(L^{1}(\mu ))\leq n.
\end{equation*}%
Hence $\overline{C}_{mNJ}^{(n)}(L^{1}(\mu ))=\overline{C}_{NJ}^{(n)}(L^{1}(%
\mu ))=n.$ Using the same arguments, by Proposition \ref{CninL} $(b)$ we
conclude that $\overline{C}_{NJ}^{(n)}(L^{\infty }(\mu ))=\overline{C}%
_{mNJ}^{(n)}(L^{\infty }(\mu ))=n.$ Similarly, by Proposition \ref{CninL} $%
(c),$ we obtain 
\begin{equation*}
\frac{1}{n}\leq \underline{C}_{NJ}^{(n)}(L^{\infty }(\mu ))\leq \underline{C}%
_{mNJ}^{(n)}(L^{\infty }(\mu ))\leq \underline{C}_{mNJ}^{(n)}(l^{\infty })=%
\frac{1}{n},
\end{equation*}%
which finishes the proof. 
\end{proof}

\begin{center}
\textbf{3. Uniformly non-}$l_{n}^{1}$ \textbf{spaces}
\end{center}

The next theorem gives some characterizations of the uniform non-$l_{n}^{1}$
property for Banach spaces. Kato, Takahashi and Hashimito proved in \cite%
{KT3} that $(X,\left\Vert \cdot \right\Vert _{X})$ is uniformly non-$%
l_{n}^{1}$ iff $\overline{C}_{NJ}^{(n)}(X)<n.$ We will extend their results.

\begin{theorem}
\label{nonln1}Let $(X,\left\Vert \cdot \right\Vert _{X})$ be a Banach space.
Then the following conditions are equivalent \newline
\hspace*{0.3in}$(a)$ $\overline{C}_{mNJ}^{(n)}(X)<n.$ \newline
\hspace*{0.3in}$(b)$ $(X,\left\Vert \cdot \right\Vert _{X})$ is uniformly
non-$l_{n}^{1}.$ \newline
\hspace*{0.3in}$(c)$ there exists $\delta \in (0,1)$ such that for any
element $(x_{1},x_{2},...,x_{n})\in B\left( l_{n}^{2}\left( X\right) \right)
,$ we have%
\begin{equation}
\min_{\theta _{j}=\pm 1}\left\Vert x_{1}+\sum\nolimits_{j=2}^{n}\theta
_{j}x_{j}\right\Vert _{X}\leq \sqrt{n}(1-\delta ).  \label{delta}
\end{equation}
\newline
\hspace*{0.3in}$(d)$ there exists $\delta \in (0,1)$ such that for any
element $(x_{1},x_{2},...,x_{n})\in S\left( l_{n}^{2}\left( X\right) \right)
,$ the inequality (\ref{delta}) is satisfied. \ 
\end{theorem}

\begin{proof} 
$(a)\Rightarrow (b).$ Suppose that $\overline{C}%
_{mNJ}^{(n)}(X)<n.$ Then 
\begin{equation*}
\frac{1}{2^{n-1}}\sum_{\theta _{j}=\pm 1}\left\Vert
x_{1}+\sum\nolimits_{j=2}^{n}\theta _{j}x_{j}\right\Vert _{X}^{2}\leq{}n\overline{%
C}_{mNJ}^{(n)}(X)
\end{equation*}%
for any $x_{1},x_{2},...,x_{n}\in S\left( X\right) .$ Since on the left hand
side we have an arithmetic mean, there is at least one sequence $(1,%
\overline{\theta }_{2},...,\overline{\theta }_{n})$ such that 
\begin{equation*}
\left\Vert x_{1}+\sum\nolimits_{j=2}^{n}\overline{\theta }%
_{j}x_{j}\right\Vert _{X}^{2}\leq n\overline{C}_{mNJ}^{(n)}(X).
\end{equation*}%
Hence 
\begin{equation*}
\min_{\theta _{j}=\pm 1}\left\Vert x_{1}+\sum\nolimits_{j=2}^{n}\theta
_{j}x_{j}\right\Vert _{X}\leq n\sqrt{\frac{\overline{C}_{mNJ}^{(n)}(X)}{n}}%
=n\left( 1-\delta \right) ,
\end{equation*}%
where $\delta =\frac{\sqrt{n}-\sqrt{\overline{C}_{mNJ}^{(n)}(X)}}{\sqrt{n}}.$
Concequently, $\left( X,\left\Vert \cdot \right\Vert _{X}\right) $ is
uniformly non-$l_{n}^{1}.$ \smallskip\ 

$(b)\Rightarrow (c).$ Assume that $\left( X,\left\Vert \cdot \right\Vert
_{X}\right) $ is uniformly non-$l_{n}^{1}.$ Let $(x_{1},x_{2},...,x_{n})\in
B\left( l_{n}^{2}\left( X\right) \right) .$ Since $\sum_{j=1}^{n}\left\Vert
x_{j}\right\Vert _{X}^{2}\leq 1,$ it follows that $\min_{1\leq j\leq
n}\left\Vert x_{j}\right\Vert _{X}\leq \frac{1}{\sqrt{n}}.$ Moreover, by the
H\"{o}lder inequality we have 
\begin{equation}
\sum_{j=1}^{n}\left\Vert x_{j}\right\Vert _{X}\leq \sqrt{n}\left(
\sum_{j=1}^{n}\left\Vert x_{j}\right\Vert _{X}^{2}\right) ^{1/2}\leq \sqrt{n}%
.  \label{H}
\end{equation}

Case $1.$ Suppose that $\frac{1}{2\sqrt{n}}<\min_{1\leq j\leq n}\left\Vert
x_{j}\right\Vert _{X}\leq \frac{1}{\sqrt{n}}.$ By the characterization of
uniform non-$l_{n}^{1}$ given in \cite{KaT} and by the inequality \eqref{H},
there is $\delta _{1}>0$ such that 
\begin{eqnarray*}
\min_{\theta _{j}=\pm 1}\left\Vert x_{1}+\sum\nolimits_{j=2}^{n}\theta
_{j}x_{j}\right\Vert _{X} &\leq &\left( 1-\frac{\delta _{1}n\min_{1\leq
i\leq n}\left\Vert x_{i}\right\Vert _{X}}{\sum_{j=1}^{n}\left\Vert
x_{j}\right\Vert _{X}}\right) \sum_{j=1}^{n}\left\Vert x_{j}\right\Vert _{X}
\\
&\leq &\left( 1-\frac{\delta _{1}\sqrt{n}}{2\sum_{j=1}^{n}\left\Vert
x_{j}\right\Vert _{X}}\right) \sum_{j=1}^{n}\left\Vert x_{j}\right\Vert _{X}
\\
&\leq &\sqrt{n}\left( 1-\frac{\delta _{1}}{2}\right) .
\end{eqnarray*}
\smallskip\ 

Case $2.$ Suppose that $0\leq \min_{1\leq j\leq n}\left\Vert x_{j}\right\Vert
_{X}\leq \frac{1}{2\sqrt{n}}.$ Let $x_{k}$ be the element on which the
minimum is taken. Then, by the H\"{o}lder inequality, we have 
\begin{eqnarray*}
\left\Vert x_{1}\pm x_{2}\pm ...\pm x_{n}\right\Vert _{X} &\leq
&\sum_{j=1,j\neq k}^{n}\left\Vert x_{j}\right\Vert _{X}+\left\Vert
x_{k}\right\Vert _{X} \\
&\leq &\sqrt{n-1}\sqrt{\sum_{j=1,j\neq k}^{n}\left\Vert x_{j}\right\Vert
_{X}^{2}}+\left\Vert x_{k}\right\Vert _{X} \\
&\leq &\sqrt{n-1}\sqrt{1-\left\Vert x_{k}\right\Vert _{X}^{2}}+\left\Vert
x_{k}\right\Vert _{X}
\end{eqnarray*}%
for any choice of signs. Define 
\begin{equation*}
f(t)=\sqrt{n-1}\sqrt{1-t^{2}}+t
\end{equation*}
for any $t\in \left[ 0,\frac{1}{2\sqrt{n}}\right] .$ By elementary calculus
we conclude that $f$ is an increasing function on the interval $\left[ 0,%
\frac{1}{2\sqrt{n}}\right] .$ Hence, the function $f(t)$ takes its highest
value on $\left[ 0,\frac{1}{2\sqrt{n}}\right] $ at the point $t=\frac{1}{2%
\sqrt{n}}.$ Thus, 
\begin{eqnarray*}
\left\Vert x_{1}\pm x_{2}\pm ...\pm x_{n}\right\Vert _{X} &\leq &\sqrt{n-1}%
\sqrt{1-\left( \frac{1}{2\sqrt{n}}\right) ^{2}}+\frac{1}{2\sqrt{n}} \\
&=&\frac{1}{2\sqrt{n}}\left( \sqrt{\left( 4n-1\right) \left( n-1\right) }%
+1\right) \\
&=&\sqrt{n}\left( 1-\frac{\left( 2n-1\right) -\sqrt{\left( 4n-1\right)
\left( n-1\right) }}{2n}\right)
\end{eqnarray*}%
for any choice of signs. Taking 
\begin{equation*}
\delta =\min \left\{ \frac{\delta _{1}}{2},\frac{\left( 2n-1\right) -\sqrt{%
\left( 4n-1\right) \left( n-1\right) }}{2n}\right\} ,
\end{equation*}
we get $(c).$ \smallskip\ 

$(c)\Rightarrow (d).$ It is obvious. \smallskip\ 

$(d)\Rightarrow (a).$ Let $(x_{1},x_{2},...,x_{n})\in S\left(
l_{n}^{2}\left( X\right) \right) .$ By the assumption $(d)$ there exists $%
\delta \in \left( 0,1\right) $ such that 
\begin{equation*}
\left\Vert x_{1}\pm x_{2}\pm ...\pm x_{n}\right\Vert _{X}\leq \sqrt{n}\left(
1-\delta \right)
\end{equation*}%
for some choice of signs. Moreover, by (\ref{H}), $\left\Vert x_{1}\pm
x_{2}\pm ...\pm x_{n}\right\Vert _{X}\leq \sqrt{n}$ for any choice of signs.
Hence, we have 
\begin{eqnarray*}
\frac{\sum_{\theta _{j}=\pm 1}\left\Vert x_{1}+\sum_{j=2}^{n}\theta
_{j}x_{j}\right\Vert _{X}^{2}}{2^{n-1}\sum_{j=1}^{n}\left\Vert
x_{j}\right\Vert _{X}^{2}} &\leq &\frac{n\left( 1-\delta \right)
^{2}+n\left( 2^{n-1}-1\right) }{2^{n-1}} \\
&=&n-\frac{\delta n\left( 2-\delta \right) }{2^{n-1}}.
\end{eqnarray*}%
By the definition of the upper $n$-th von Neumann-Jordan constant $\overline{%
C}_{NJ}^{(n)}(X)$ and Proposition \ref{D}, we conclude 
\begin{equation*}
\overline{C}_{mNJ}^{(n)}(X)\leq \overline{C}_{NJ}^{(n)}(X)\leq n-\frac{%
\delta n\left( 2-\delta \right) }{2^{n-1}}<n,
\end{equation*}%
which finishes the proof. 
\end{proof}

By Theorem 3 and the definition of $B$-convexity, we get immediately
\smallskip\ 

\begin{corollary}
\label{B}A Banach space $(X,\left\Vert \cdot \right\Vert _{X})$ is B-convex
if and only if there is $n\geq 2$ $(n\in N)$ such that $\overline{C}%
_{mNJ}^{(n)}(X)<n.$
\end{corollary}

Notice that $\overline{C}_{mNJ}^{(n)}(X)$ is not equal to $\overline{C}%
_{NJ}^{(n)}(X)$ in general (for $n=2$ see \cite{MS}).

\begin{corollary}
\label{CC}$\overline{C}_{mNJ}^{(n)}(X)=n$ if and only if $\overline{C}%
_{NJ}^{(n)}(X)=n.$
\end{corollary}

\begin{proof}
Since $(X,\left\Vert \cdot \right\Vert _{X})$ is uniformly
non-$l_{n}^{1}$ iff $\overline{C}_{NJ}^{(n)}(X)<n,$ it follows, by Theorem %
\ref{nonln1}, that $\overline{C}_{NJ}^{(n)}(X)<n$ iff $\overline{C}%
_{mNJ}^{(n)}(X)<n.$ Hence, by Proposition \ref{propCn} \ $(a),$ we get the
thesis.
\end{proof}

\begin{remark}
	Let us notice that the above corollary can be reformulated equivalently as follows 
	\begin{equation*}
	\overline{C}_{mNJ}^{(n)}(X)<n \quad\textnormal{ if and only if }\quad\overline{C}_{NJ}^{(n)}(X)<n.
	\end{equation*}
\end{remark}

\begin{center}
\textbf{4. Upper and lower }$n$\textbf{-th von Neumann-Jordan constant for }$%
L^{p}$\textbf{-spaces}
\end{center}

Now we will calculate the upper $n$-th von Neumann-Jordan constant for
Lebesgue spaces $L^{p}(\mu )$ and $l^{p}$ $(1<p<\infty )$. To prove the next
lemma, we will apply the following results given by Figiel, Iwaniec and Pe\l %
czy\'{n}ski in \cite{FIP}. Namely, for arbitrary scalars $%
c_{1},c_{2},...,c_{n}$ and $2<p<\infty $ we have 
\begin{equation}
\int_{0}^{1}\left\vert \sum_{j=1}^{n}c_{j}r_{j}(t)\right\vert ^{p}dt\leq
n^{-1}\int_{0}^{1}\left\vert \sum_{j=1}^{n}r_{j}(t)\right\vert
^{p}dt\sum_{j=1}^{n}\left\vert c_{j}\right\vert ^{p},  \label{fig}
\end{equation}%
where $r_{1},r_{2},...,r_{n}$ $(n=1,2,...)$ are Rademacher functions, that
is $r_{n}(t)=$ sign$\left( \sin 2^{n}\pi t\right) .$

Let $\left\lfloor \cdot \right\rfloor :\mathbb{R}\rightarrow \mathbb{Z}$ be
the floor function, i.e. $\left\lfloor x\right\rfloor =\max \left\{ m\in 
\mathbb{Z}:m\leq x\right\} $ for any $x\in \mathbb{R}.$

\begin{lemma}
\label{ineqlp}Let $2<p<\infty $ and $X=L^{p}(\mu )$ or $X=l^{p}.$ Then 
\begin{equation*}
\sum_{\theta _{j}=\pm 1}\left\Vert x_{1}+\sum_{j=2}^{n}\theta
_{j}x_{j}\right\Vert _{X}^{p}\leq n^{-1}\sum\limits_{k=0}^{\left\lfloor
n/2\right\rfloor }\binom{n}{k}\left( n-2k\right)
^{p}\sum_{j=1}^{n}\left\Vert x_{j}\right\Vert _{X}^{p}
\end{equation*}%
for any $x_{1},x_{2},...,x_{n}\in X$ and any integer $n\geq 1.$ \smallskip\ 
\end{lemma}

\begin{proof}
Fix an integer $n\geq 1.$ Notice that 
\begin{equation*}
\int_{0}^{1}\left\vert \sum_{j=1}^{n}c_{j}r_{j}(t)\right\vert
^{p}dt=2^{1-n}\sum_{\theta _{j}=\pm 1}\left\vert c_{1}+\sum_{j=2}^{n}\theta
_{j}c_{j}\right\vert ^{p}
\end{equation*}%
for any scalars $c_{1},c_{2},...,c_{n}.$ On the other hand, it can be proved
elementarily that 
\begin{equation*}
\int_{0}^{1}\left\vert \sum_{j=1}^{n}r_{j}(t)\right\vert
^{p}dt=2^{1-n}\sum\limits_{k=0}^{\left\lfloor n/2\right\rfloor }\binom{n}{k}%
\left( n-2k\right) ^{p}.
\end{equation*}%
Suppose that $x_{k}=\left( t_{i}^{(k)}\right) _{i=1}^{\infty }\in l^{p}$ for 
$k=1,2,...,n.$ By the inequality (\ref{fig}), we get 
\begin{equation*}
\sum_{\theta _{j}=\pm 1}\left\vert t_{i}^{(1)}+\sum_{j=2}^{n}\theta
_{j}t_{i}^{(j)}\right\vert ^{p}\leq n^{-1}\sum\limits_{k=0}^{\left\lfloor
n/2\right\rfloor }\binom{n}{k}\left( n-2k\right)
^{p}\sum_{j=1}^{n}\left\vert t_{i}^{(j)}\right\vert ^{p}
\end{equation*}%
for any $i\in \mathbb{N}.$ Summing by sides from $i=1$ to $\infty $ and
reversing the order of summation, we obtain the thesis. Similarly, for $%
X=L^{p}(\mu )$ take $x_{1},x_{2},...,x_{n}\in L^{p}(\mu ).$ Then, by the
inequality (\ref{fig}), we get 
\begin{equation*}
\sum_{\theta _{j}=\pm 1}\left\vert x_{1}(t)+\sum_{j=2}^{n}\theta
_{j}x_{j}(t)\right\vert ^{p}\leq n^{-1}\sum\limits_{k=0}^{[n/2]}\binom{n}{k}%
\left( n-2k\right) ^{p}\sum_{j=1}^{n}\left\vert x_{j}(t)\right\vert ^{p}
\end{equation*}%
for almost every $t\in \Omega .$ Integrating by sides this inequality and
reversing the order of summation and integration, we obtain the desired
inequality.
\end{proof}

\begin{theorem}
\label{mCJ}Let $1\leq p<\infty $ and $X=L^{p}(\mu )$ or $X=l^{p}.$ Then 
\begin{equation*}
\overline{C}_{mNJ}^{(n)}(X)=\left\{ 
\begin{array}{lll}
n^{\frac{2}{p}-1} & \text{\textit{if}} & 1\leq p\leq 2\text{ and }\dim X\geq
n,\text{ } \\ 
n^{-1}\left( 2^{1-n}\sum\limits_{k=0}^{\left\lfloor n/2\right\rfloor }\binom{%
n}{k}\left( n-2k\right) ^{p}\right) ^{\frac{2}{p}} & \text{\textit{if}} & 
2<p<\infty \text{ and }\dim X\geq 2^{n-1}.%
\end{array}%
\right.
\end{equation*}
\end{theorem}

\begin{proof}
Case 1. Let $1\leq p\leq 2.$ By Theorem 3 from \cite{KT3} and
Proposition \ref{propCn} $(a),$ for all $n\geq 2,$ we have 
\begin{equation*}
\overline{C}_{mNJ}^{(n)}(X)\leq \overline{C}_{NJ}^{(n)}(X)=n^{\frac{2}{p}%
-1}.
\end{equation*}%
whenever $X=L^{p}(\mu )$ or $X=l^{p}.$ The opposite inequality follows
immediately from Proposition \ref{CninL} $(a)$ whenever $X=l^{p}$ with $\dim
l^{p}\geq n.$

Now consider $X=L^{p}(\mu ).$ Let $A\subset \Omega $ be a set of positive
finite measure. Divide the set $A$ into $n$ pairwise disjoint subsets $%
A_{1}, $ $A_{2},...,A_{n}$ such that $\bigcup\limits_{i=1}^{n}A_{i}=A$ and $%
\mu (A_{i})=\frac{1}{n}\mu (A).$ Define $z_{i}=\mu (A_{i})^{-1/p}\chi
_{A_{i}}$ for $i=1,2,...,n.$ Then 
\begin{equation*}
\left\Vert z_{i}\right\Vert _{L^{p}}=\left( \int\limits_{A_{i}}\left( \mu
(A_{i})^{-1/p}\right) ^{p}d\mu \right) ^{\frac{1}{p}}=1
\end{equation*}%
for any $i\in \{1,2,...,n\}$ and 
\begin{eqnarray*}
\overline{C}_{mNJ}^{(n)}(L^{p}) &\geq &C^{(n)}(z_{1},z_{2},...,z_{n}) \\
&=&\frac{1}{n2^{n-1}}\sum_{\theta _{j}=\pm 1}\left\Vert
z_{1}+\sum_{j=2}^{n}\theta _{j}z_{j}\right\Vert _{L^{p}}^{2} \\
&=&\frac{1}{n2^{n-1}}\sum_{\theta _{j}=\pm 1}\left\Vert \mu
(A_{1})^{-1/p}\chi _{A_{1}}+\sum_{j=2}^{n}\theta _{j}\mu (A_{j})^{-1/p}\chi
_{A_{j}}\right\Vert _{L^{p}}^{2} \\
&=&\frac{1}{n2^{n-1}}2^{n-1}\left\Vert \left( \frac{1}{n}\mu (A)\right)
^{-1/p}\chi _{A}\right\Vert _{L^{p}}^{2} \\
&=&\frac{1}{n}\left[ \mu (A)\left( \frac{1}{n}\mu (A)\right) ^{-1}\right] ^{%
\frac{2}{p}}=n^{\frac{2}{p}-1}.
\end{eqnarray*}%
Hence $\overline{C}_{mNJ}^{(n)}(L^{p})=\overline{C}_{NJ}^{(n)}(L^{p})=n^{%
\frac{2}{p}-1},$ whenever $1\leq p\leq 2.$ \medskip\ 

Case 2. Let $2<p<\infty $ and $X=L^{p}(\mu )$ or $X=l^{p}.$ By the H\"older-Rogers inequality for $p>2$ and by Lemma \ref{ineqlp}, we have 
\begin{eqnarray}
\sum_{\theta _{j}=\pm 1}\left\Vert x_{1}+\sum_{j=2}^{n}\theta
_{j}x_{j}\right\Vert _{X}^{2} &\leq &2^{2(n-1)\left( \frac{1}{2}-\frac{1}{p}%
\right) }\left( \sum_{\theta _{j}=\pm 1}\left\Vert
x_{1}+\sum_{j=2}^{n}\theta _{j}x_{j}\right\Vert _{X}^{p}\right) ^{\frac{2}{p}%
}  \notag \\
&\leq &2^{(n-1)\left( \frac{p-2}{p}\right) }\left(
n^{-1}\sum\limits_{k=0}^{\left\lfloor n/2\right\rfloor }\binom{n}{k}\left(
n-2k\right) ^{p}\sum_{j=1}^{n}\left\Vert x_{j}\right\Vert _{X}^{p}\right) ^{%
\frac{2}{p}}  \label{Clr}
\end{eqnarray}%
for any $x_{1},x_{2},...,x_{n}\in X.$ Assume that $x_{1},x_{2},...,x_{n}\in
S(X).$ Then, by inequality (\ref{Clr}), we have 
\begin{eqnarray*}
C^{(n)}(x_{1},x_{2},...,x_{n}) &=&\frac{\sum_{\theta _{j}=\pm 1}\left\Vert
x_{1}+\sum_{j=2}^{n}\theta _{j}x_{j}\right\Vert _{X}^{2}}{n2^{n-1}} \\
&\leq &\frac{1}{n2^{n-1}}2^{(n-1)\left( \frac{p-2}{p}\right) }\left(
n^{-1}\sum\limits_{k=0}^{\left\lfloor n/2\right\rfloor }\binom{n}{k}\left(
n-2k\right) ^{p}\sum_{j=1}^{n}\left\Vert x_{j}\right\Vert _{X}^{p}\right) ^{%
\frac{2}{p}} \\
&=&\frac{1}{n}2^{(n-1)\left( \frac{p-2}{p}-1\right) }\left(
\sum\limits_{k=0}^{\left\lfloor n/2\right\rfloor }\binom{n}{k}\left(
n-2k\right) ^{p}\right) ^{\frac{2}{p}} \\
&=&n^{-1}\left( 2^{1-n}\sum\limits_{k=0}^{[n/2]}\binom{n}{k}\left(
n-2k\right) ^{p}\right) ^{\frac{2}{p}},
\end{eqnarray*}%
whence 
\begin{equation}
\overline{C}_{mNJ}^{(n)}(X)\leq n^{-1}\left(
2^{1-n}\sum\limits_{k=0}^{\left\lfloor n/2\right\rfloor }\binom{n}{k}\left(
n-2k\right) ^{p}\right) ^{\frac{2}{p}}.  \label{CmNJup}
\end{equation}%
Let the matrix $A_{n}$ be defined as in the proof of Proposition \ref{D} $%
(b).$ Denote by $y_{i}$ column $i$ of the matrix $A_{n}$ $(i=1,2,...,n).$
For any $i\in \left\{ 1,2,...,n\right\} $ define 
\begin{equation*}
z_{i}=\frac{1}{\left( 2^{n-1}\right) ^{1/p}}y_{i}^{T},
\end{equation*}%
where $y_{i}^{T}$ denotes the transpose of the column $y_{i}.$ Then 
\begin{equation*}
\left\Vert z_{i}\right\Vert _{l_{2^{n-1}}^{p}}=\left(
\sum\nolimits_{m=1}^{2^{n-1}}\left( \frac{1}{\left( 2^{n-1}\right) ^{1/p}}%
\right) ^{p}\right) ^{1/p}=1
\end{equation*}%
for any $i\in \left\{ 1,2,...,n\right\} .$ Hence $z_{1},z_{2},...,z_{n}\in
S\left( l_{2^{n-1}}^{p}\right) .$ For any element $%
x=(t_{1},t_{2},...,t_{2^{n-1}})\in l_{2^{n-1}}^{p}$ denote by $x^{\ast }$
its non-increasing rearrangement, i.e. a non-increasing sequence obtained
from $\left\{ \left\vert t_{i}\right\vert \right\} _{i=1}^{2^{n-1}}$ by a
suitable permutation of the integers. Notice that for all sequences $%
(1,\theta _{2},...,\theta _{n})$ such that $\theta _{j}=\pm 1,$ $%
(j=2,3,...,n)$ the non-increasing rearrangements $\left(
z_{1}+\sum_{j=2}^{n}\theta _{j}z_{j}\right) ^{\ast }$ coincide. Denoting by $%
\left( v_{1},v_{2},...,v_{2^{n-1}}\right) $ the non-increasing sequence such
that 
\begin{equation*}
\left( z_{1}+\sum\nolimits_{j=2}^{n}\theta _{j}z_{j}\right) ^{\ast }=\left(
v_{1},v_{2},...,v_{2^{n-1}}\right)
\end{equation*}%
for any sequences $(1,\theta _{2},...,\theta _{n}).$ Hence 
\begin{eqnarray*}
C^{(n)}(x_{1},x_{2},...,x_{n}) &=&\frac{\sum_{\theta _{j}=\pm 1}\left\Vert
z_{1}+\sum_{j=2}^{n}\theta _{j}z_{j}\right\Vert _{l_{2^{n-1}}^{p}}^{2}}{%
n2^{n-1}} \\
&=&\frac{1}{n}\left\Vert \left( v_{1},v_{2},...,v_{2^{n-1}}\right)
\right\Vert _{l_{2^{n-1}}^{p}}^{2}.
\end{eqnarray*}%
Notice that $v_{1}=\frac{n}{\left( 2^{n-1}\right) ^{1/p}},$ $v_{l}=\frac{n-2k%
}{\left( 2^{n-1}\right) ^{1/p}}$ for $\binom{n}{k}$ subsequent integers $l,$ 
$\left( k=1,2,...,\left\lfloor n/2\right\rfloor \right) .$ Consequently, 
\begin{eqnarray*}
\overline{C}_{mNJ}^{(n)}\left( l_{2^{n-1}}^{p}\right) &\geq
&C^{(n)}(x_{1},x_{2},...,x_{n}) \\
&=&\frac{1}{n}\left( \sum\limits_{k=0}^{\left\lfloor n/2\right\rfloor }%
\binom{n}{k}\left( \frac{n-2k}{\left( 2^{n-1}\right) ^{1/p}}\right)
^{p}\right) ^{2/p} \\
&=&n^{-1}\left( 2^{1-n}\sum\limits_{k=0}^{\left\lfloor n/2\right\rfloor }%
\binom{n}{k}\left( n-2k\right) ^{p}\right) ^{\frac{2}{p}}.
\end{eqnarray*}%
Since $l_{2^{n-1}}^{p}$ can be embedded isometrically in any $l^{p}$ with $%
\dim l^{p}\geq 2^{n-1},$ by inequality (\ref{CmNJup}) applied for $X=l^{p}$,
it follows that 
\begin{equation*}
\overline{C}_{mNJ}^{(n)}\left( l_{2^{n-1}}^{p}\right) =\overline{C}%
_{mNJ}^{(n)}\left( l^{p}\right) =n^{-1}\left(
2^{1-n}\sum\limits_{k=0}^{\left\lfloor n/2\right\rfloor }\binom{n}{k}\left(
n-2k\right) ^{p}\right) ^{\frac{2}{p}}
\end{equation*}%
whenever $2<p\leq \infty $ and $\dim l^{p}\geq 2^{n-1}.$

Since $L^{p}(\mu )$ contains an isometric copy of $l_{2^{n-1}}^{p},$ by
inequality (\ref{CmNJup}), we obtain the thesis for $X=L^{p}(\mu ),$ which
completes the proof.
\end{proof}

Haagerup \cite{Haa} proved that the best type $(2,p)$ constant in the
Khinthine inequality for $2\leq p<\infty $ is $B_{p}=\sqrt{2}\left( \frac{%
\Gamma \left( \frac{p+1}{2}\right) }{\sqrt{\pi }}\right) ^{\frac{1}{p}}.$
Kato, Takahashi and Hashimoto proved in \cite{KT3} that $\overline{C}%
_{NJ}^{(n)}\left( X\right) \leq \min \left\{ n^{\frac{2}{q}%
-1},B_{p}^{2}\right\} .$ Combining this result with Theorem \ref{mCJ}, we
get two hand side estimation of upper von Neuman-Jordan constant for
Lebesgue spaces with $p\in (2,\infty ).$

\begin{corollary}
\label{CJLest}Let $2<p<\infty ,$ $q$ be conjugate to $p$ and $X=L^{p}(\mu )$
or $X=l^{p}.$ If $\dim X\geq 2^{n-1},$ then 
\begin{equation*}
n^{-1}\left( 2^{1-n}\sum\limits_{k=0}^{\left\lfloor n/2\right\rfloor }\binom{%
n}{k}\left( n-2k\right) ^{p}\right) ^{\frac{2}{p}}\leq \overline{C}%
_{NJ}^{(n)}\left( X\right) \leq \min \left\{ n^{\frac{2}{q}%
-1},B_{p}^{2}\right\} .
\end{equation*}
\end{corollary}

\begin{proof}
The left hand side inequality follows immediately from Theorem \ref%
{mCJ}. Namely, 
\begin{equation*}
n^{-1}\left( 2^{1-n}\sum\limits_{k=0}^{\left\lfloor n/2\right\rfloor }\binom{%
n}{k}\left( n-2k\right) ^{p}\right) ^{\frac{2}{p}}=\overline{C}%
_{mNJ}^{(n)}\left( X\right) \leq \overline{C}_{NJ}^{(n)}\left( X\right) ,
\end{equation*}%
whenever $2<p<\infty ,$ $X=L^{p}(\mu )$ or $X=l^{p}$ and $\dim X\geq
2^{n-1}. $ The right hand side inequality was proved in \cite{KT3}.
\end{proof}

\begin{corollary}
\label{mCJL}Let $2\leq p\leq \infty $ and $X=L^{p}(\mu )$ or $X=l^{p}.$ If $%
\dim X\geq $ $n,$ then $\underline{C}_{NJ}^{(n)}(X)=n^{\frac{2}{p}%
-1}.$ \medskip\ 
\end{corollary}

\begin{proof} 
Let $2\leq p\leq \infty $ and $q$\ denote the conjugate
number of $p.$ By Corollary \ref{Dual} and Theorem \ref{mCJ}, we have 
\begin{equation*}
\underline{C}_{NJ}^{(n)}(L^{p}(\mu ))\geq \frac{1}{\overline{C}%
_{NJ}^{(n)}(\left( L^{p}(\mu )\right) ^{\ast })}=\frac{1}{\overline{C}%
_{NJ}^{(n)}(L^{q}(\mu ))}=n^{1-\frac{2}{q}}=n^{\frac{2}{p}-1}.
\end{equation*}%
On the other hand, taking the canonical basis $\left\{ e_{i}\right\}
_{i=1}^{n}$ in $l_{n}^{p}$ we have 
\begin{eqnarray*}
C^{(n)}\left( e_{1},e_{2},...,e_{n}\right) &=&\frac{\sum_{\theta _{j}=\pm
1}\left\Vert e_{1}+\sum_{j=2}^{n}\theta _{j}e_{j}\right\Vert _{l_{n}^{p}}^{2}%
}{2^{n-1}\sum_{j=1}^{n}\left\Vert e_{j}\right\Vert _{l_{n}^{p}}^{2}} \\
&=&\frac{2^{n-1}n^{\frac{2}{p}}}{2^{n-1}n}=n^{\frac{2}{p}-1}.
\end{eqnarray*}%
Hence, by the definition of lower $n$-th von Neumann-Jordan constant, we get 
\begin{equation*}
\underline{C}_{NJ}^{(n)}(L^{p}(\mu ))\leq \underline{C}%
_{NJ}^{(n)}(l_{n}^{p})\leq n^{\frac{2}{p}-1}
\end{equation*}%
whenever $\dim L^{p}(\mu )\geq n.$ Combining the both inequalities, we get
the thesis.
\end{proof}

\begin{remark}
It is known that $C_{NJ}(X)=C_{NJ}(X^{\ast })$ (see \cite{KT2}) and in
general $\overline{C}_{NJ}^{(n)}(X)\neq \overline{C}_{NJ}^{(n)}(X^{\ast })$
for $n\geq 3$ (see \cite{KT3}). Theorem \ref{mCJ} shows that $\overline{C}%
_{mNJ}^{(2)}(X)=\overline{C}_{mNJ}^{(2)}(X^{\ast })$ whenever $X=L^{p}(\mu )$
or $X=l^{p}$ with $\dim X\geq 2^{n-1}.$ Really, fix $1<p<2$ and consider $%
X=L^{p}(\mu )$ or $X=l^{p}$ with $\dim X\geq 2^{n-1}.$ Let $q$ be conjugate
to $p.$ Then $q>2$ and $X^{\ast }=L^{q}(\mu )$ or $X^{\ast }=l^{q}$ with $%
\dim X^{\ast }\geq 2^{n-1},$ respectively. Applying Theorem \ref{mCJ} for $%
n=2,$ we have 
\begin{equation*}
\overline{C}_{mNJ}^{(2)}(X)=2^{\frac{2}{p}-1}.
\end{equation*}%
Since $q=\frac{p}{p-1}>2,$ it follows from Theorem \ref{mCJ} that 
\begin{eqnarray*}
\overline{C}_{mNJ}^{(2)}(X^{\ast }) &=&\frac{1}{2}\left( \frac{1}{2}%
\sum\limits_{k=0}^{1}\binom{2}{k}\left( 2-2k\right) ^{\frac{p}{p-1}}\right)
^{\frac{2(p-1)}{p}} \\
&=&\frac{1}{2}\left( 2^{\frac{p}{p-1}-1}\right) ^{\frac{2(p-1)}{p}}=2^{\frac{%
2}{p}-1},
\end{eqnarray*}%
whence $\overline{C}_{mNJ}^{(2)}(X)=\overline{C}_{mNJ}^{(2)}(X^{\ast }).$ 
\newline
\hspace*{0.3in}The equality $\overline{C}_{mNJ}^{(n)}(X)=\overline{C}%
_{mNJ}^{(n)}(X^{\ast })$ does not hold in general for $n\geq 3.$ By Remark $%
9 $\nolinebreak $(ii)$ in \cite{KT3} and Theorem \ref{mCJ}, we have 
\begin{equation*}
\overline{C}_{mNJ}^{(n)}(X^{\ast })\leq \overline{C}_{NJ}^{(n)}(X^{\ast })<%
\overline{C}_{NJ}^{(n)}(X)=n^{\frac{2}{p}-1}=\overline{C}_{mNJ}^{(n)}(X),
\end{equation*}%
whence $\overline{C}_{mNJ}^{(n)}(X)\not=\overline{C}_{mNJ}^{(n)}(X^{\ast })$
for $n\geq 3.$
\end{remark}

\bigskip\ 

{\small Maciej CIESIELSKI, Institute of Mathematics, Pozna\'{n} University
of Tech\-no\-lo\-gy, Piotrowo 3A, 60-965 POZNA\'{N}, Poland}

{\small email: maciej.ciesielski@put.poznan.pl;}

{\small Ryszard P\L UCIENNIK, Institute of Mathematics, Pozna\'{n}
University of Tech\-no\-lo\-gy, Piotrowo 3A, 60-965 POZNA\'{N}, Poland }

{\small email: ryszard.pluciennik@put.poznan.pl;}


\begin{thebibliography}{99}
\bibitem{AK} A.G. Aksoy and M.A. Khamsi, \textit{Nonstandard Methods in
Fixed Point Theory, }Springer Verlag 1990.

\bibitem{AMP} J. Alonso, P. Martin and P. L. Papini, \textit{Wheeling around
von Neumann-Jordan constant in Banach Spaces}, Studia Math., \textbf{188 }%
(2008), 135-150.

\bibitem{B} A. Beck, \textit{On convexity condition in Banach spaces and the
strong law of large numbers}, Proc. Amer. Math. Soc. \textbf{13} (1962),
329-334.

\bibitem{BL} J. Berg and J. L\"{o}fstr\"{o}m, Interpolation spaces,
Springer-Verlag, Berlin-Heidelberg-New York, 1976.

\bibitem{C} J.A. Clarkson, \textit{Uniformly convex spaces}, Trans. Amer.
Math. Soc. \textbf{40} (1936), 39\i 6-414.

\bibitem{C1} J.A. Clarkson, \textit{The von Neumann-Jordan constant for the
Lebesgue space}, Ann. of Math. \textbf{38} (1937), 114-115.

\bibitem{FIP} T. Figiel, T. Iwaniec and A. Pe\l czy\'{n}ski, \textit{%
Computing norm and critical exponents of some operators in} $L^{p}$-\textit{%
spaces,} Studia Math. \textbf{79} (1984), 227-274.

\bibitem{Gao} J. Gao, \textit{A pythagorean approach in Banach spaces}, J.
Inequal. Appl., 2006, Art. ID 94982, 1-11.

\bibitem{GS} J. Gao and S. Saejung, \textit{Some geometric measures of
spheres in Banach spaces}, Appl. Math. Comput., \textbf{214 }(2009),
102--107.

\bibitem{GJ} D.P. Giesy and R.C. James, \textit{Uniformly non}-$l_{n}^{(1)}$ 
\textit{and B-convex spaces}, Studia Math. \textbf{48} (1973), 61-69.

\bibitem{GHO} R. Grz\c{a}\'{s}lewicz, H. Hudzik and W. Orlicz, \textit{%
Uniformly non}-$l_{n}^{(1)}$ \textit{property in some normed spaces},
Bull.Acad. Polon. Sci. Math. 34, 3-4 (1986), 161-171.

\bibitem{Haa} U. Haagerup, \textit{The best constants in the Khinthine
inequality}, Studia Math. \textbf{70} (1982), 231-283.

\bibitem{Hu1} H. Hudzik, \textit{Some classes of uniformly non}-square 
\textit{Orlicz-Bochner spaces}, Comm. Math. Univ. Carolinae \textbf{26}
(1985), 269-274.

\bibitem{Hu} H. Hudzik, \textit{Uniformly non}-$l_{n}^{(1)}$ \textit{Orlicz
spaces with Luxemburg norm}, Studia Math. \textbf{81} (1985), 271-284.

\bibitem{HuK} H. Hudzik and A. Kami\'{n}ska, \textit{On uniformly
convexifiable and }$B$-\textit{convex Musielak-Orlicz spaces,} Comment.
Math. \textbf{25} (1985), 59-75.

\bibitem{Ja} R.C. James, \textit{Uniformly non-square Banach spaces}, Ann.
of Math. \textbf{80} (1964), 542-550.

\bibitem{JN} P. Jordan and J. von Neumann, \textit{On inner product in
linear metric spaces,} Ann. of Math. \textbf{36} (1935), 719-723.

\bibitem{KaT} A. Kami\'{n}ska and B. Turett, \textit{Uniformly non-}$%
l_{n}^{1}$ \textit{Orlicz-Bochner spaces,} Bull. Pol. Acad.Sci. Math. \ 
\textbf{35.3-4} (1987), 211-218.

\bibitem{K1} M. Kato, \textit{Generalized Clarkson inequalities and the norm
of Littlewood matrices, }Math. Nachr. \textbf{114 }(1983), 163-170.

\bibitem{KMT} M. Kato, L.Maligranda and Y. Takahashi, \textit{On James and
Jordan-von Neumann constants and the normal structure coefficient of Banach
spaces,} Studia Math. \textbf{144} (3) (2001), 275-296.

\bibitem{KT2} M. Kato and Y. Takahashi, \textit{On the }v\textit{on
Neumann-Jordan constant for Banach spaces,} Proc. Amer. Math. Soc. \textbf{%
125} (1997), 1055-1062.

\bibitem{KT3} M. Kato, Y. Takahashi and K. Hashimoto, \textit{On }$n$\textit{%
-th }v\textit{on Neumann-Jordan constants for Banach spaces,} Bull. Kyushu
Inst. Tech. Pure. Appl. Math. \textbf{45} (1998), 25-33.

\bibitem{LT} J. Lindenstrauss and L. Tzafriri, \textit{Classical Banach
spaces II: Function Spaces}, Ergebnisse der Mathematik und ihrer
Grenzgebiete \textbf{97,} Springer-Verlag, Berlin(1979).

\bibitem{MS} H. Mizuguchi and K. -S. Saito, \textit{Some geometric constants
of absolute normalized norms on} $\mathbb{R}^{2}$, Ann. Funct. Anal., 
\textbf{2 }(2011), 22--33.

\bibitem{TK} Y. Takahashi and M. Kato, \textit{Von Neumann-Jordan constant
and uniformly non-square Banach Spaces, }Nihonkai Math. J. 9 (1998), 155-169.

\bibitem{WJ} J. Wang, \textit{On non}-$l_{n}^{(1)}$ \textit{constant and} $n$%
-\textit{th von Neumann-Jordan constants for Orlicz spaces}, Comment.
Mathematicae 44 (2) (2004), 187-197.
\end{thebibliography}
\end{document}